\documentclass[12pt]{amsart}
\usepackage{amssymb,mathrsfs, url}
\usepackage[matrix,arrow,curve]{xy}
\usepackage{setspace}
\newtheorem{theorem}[equation]{Theorem}
\newtheorem*{theorem*}{Theorem}
\newtheorem{lemma}[equation]{Lemma}
\newtheorem{corollary}[equation]{Corollary}
\newtheorem{conjecture}[equation]{Conjecture}
\newtheorem{question}[equation]{Question}
\newtheorem{proposition}[equation]{Proposition}

\theoremstyle{definition}

\newtheorem{definition}[equation]{Definition}
\newtheorem*{definition*}{Definition}

\theoremstyle{remark}
\newtheorem{remark}[equation]{Remark}

\makeatletter\@addtoreset{equation}{section}
\makeatletter\@addtoreset{section}{part}

\makeatother

\def \K {\mathcal{K}}
\def \CC {\mathcal{C}}

\def \FF {\mathcal{F}}
\def \RR {\mathcal{R}}
\def \P {\mathbb{P}}
\def \R {\mathbb{R}}
\def \Q {\mathbb{Q}}
\def \C {\mathbb{C}}

\def \qlin {\sim_{\Q}}

\newcommand{\Cr}{\operatorname{Cr}}

\def \Bs {\mathrm{Bs}\,}
\def \Bir {\mathrm{Bir}}
\def \Aut {\mathrm{Aut}}

\def \GL {\mathrm{GL}}

\newcommand{\N}{\operatorname{N}}

\def \le {\leqslant}

\title{Jordan property for Cremona groups}

\author{Yuri Prokhorov}
\author{Constantin Shramov}

\address{
Steklov Institute of Mathematics, 8 Gubkina street, Moscow 119991, Russia
\newline
and
\newline\indent
Laboratory of Algebraic Geometry, GU-HSE, 7 Vavilova street,
Moscow 117312, \mbox{Russia}}

\email{prokhoro@gmail.com}

\email{costya.shramov@gmail.com}

\begin{document}

\begin{abstract}
Assuming Borisov--Alexeev--Borisov conjecture, we
prove that there is a constant $J=J(n)$ such that
for any rationally connected variety $X$ of dimension $n$
and any
finite subgroup $G\subset\Bir(X)$ there exists
a normal abelian subgroup $A\subset G$ of index at most $J$.
In particular, we obtain that the Cremona group $\Cr_3=\Bir(\P^3)$ enjoys
the Jordan property.
\end{abstract}

\maketitle


\section{Introduction}
\label{section:intro}
Unless explicitly stated otherwise, all varieties below are assumed to be defined over
an algebraically closed field~$\Bbbk$ of characteristic~$0$.

The \textit{Cremona group} $\Cr_n(\Bbbk)$ is the group of birational transformations of
the projective space $\mathbb P^n$.
The group $\Cr_2(\Bbbk)$ and its subgroups have been
a subject of research for many years (see~\cite{Deserti2012},
\cite{Dolgachev-Iskovskikh}, \cite{Serre-2008-2009}
and references therein).
The main philosophical
observation is that this group is very large and it is
``very far'' from being a linear group. However, the system of its finite
subgroups seems more accessible, and in particular
happens to enjoy many features of finite subgroups
in $\GL_n(\Bbbk)$ (which are actually not obvious even for
the subgroups of~$\GL_n(\Bbbk)$).

\begin{theorem}[C.\,Jordan, {see e.\,g.~\cite[Theorem 36.13]{Curtis-Reiner-1962}}]
\label{theorem:Jordan}
There is a constant $I=I(n)$ such that for any
finite subgroup $G\subset\GL_n(\C)$
there exists
a normal abelian subgroup $A\subset G$ of index at most $I$.
\end{theorem}

This leads to the following definition
(cf.~\cite[Definition~2.1]{Popov2011}).

\begin{definition}
\label{definition:Jordan}
A group $\Gamma$ is called \emph{Jordan}
(alternatively, we say
that~$\Gamma$ \emph{has
Jordan property})
if there is a constant $J$ such that
for any finite subgroup $G\subset\Gamma$ there exists
a normal abelian subgroup $A\subset G$ of index at most $J$.
\end{definition}

Theorem~\ref{theorem:Jordan} implies that
all linear algebraic groups
over an arbitrary field $\Bbbk$ with~\mbox{$\mathrm{char}(\Bbbk)=0$}
are Jordan.
The same question is of interest
for other ``large'' groups, especially those
that are more accessible for study on the level
of finite subgroups than on the global level,
in particular, for the groups
of birational selfmaps of algebraic varieties.
A complete answer is known in dimension at most~$2$.
Moreover, already in dimension~$2$ it appears to be non-trivial,
i.\,e. there are surfaces providing
a positive answer to the question, as well as surfaces
providing a negative answer.

First of all, the automorphism group of any curve
is Jordan.
The Cremona group of rank $2$
is Jordan too.

\begin{theorem}[J.-P.\,Serre
{\cite[Theorem~5.3]{Serre2009}, \cite[Th\'eor\`eme~3.1]{Serre-2008-2009}}]
The Cremona group~\mbox{$\Cr_2(\Bbbk)$} is Jordan.
\end{theorem}

On the other hand, starting from dimension $2$ one can construct varieties
with non-Jordan groups of birational selfmaps.

\begin{theorem}[Yu.\,Zarhin {\cite{Zarhin10}}]
\label{theorem:Zarhin}
Suppose that $X\cong E\times\P^1$, where $E$ is an abelian
variety of dimension $\dim(E)>0$.
Then the group $\Bir(X)$ is not Jordan.
\end{theorem}

In any case, in dimension~$2$ it is possible to give a complete
classification of surfaces with Jordan groups
of birational automorphisms.

\begin{theorem}[V.\,Popov {\cite[Theorem~2.32]{Popov2011}}]
Let $S$ be a surface. Then the group
$\Bir(S)$ is Jordan if and only if $S$ is not
birational to $E\times\P^1$, where $E$ is an
elliptic curve.
\end{theorem}

Somehow, in higher dimensions the answer remained
unknown even for a more particular question.

\begin{question}[J.-P.\,Serre {\cite[6.1]{Serre2009}}]
\label{question:Jordan}
Is the group $\Cr_n(\Bbbk)$ Jordan?
\end{question}

Question~\ref{question:Jordan} asks about some kind of
boundedness related to the geometry of rational
varieties. It is not a big surprise that it appears
to be related to another ``boundedness conjecture'',
that is a particular case of the well-known
Borisov--Alexeev--Borisov conjecture (see~\cite{Borisov-1996}).

\begin{conjecture}
\label{conjecture:BAB}
For a given
positive integer $n$, Fano varieties
of dimension $n$ with terminal singularities
are bounded, i.\,e. are contained in a finite
number of algebraic families.
\end{conjecture}

Note that if Conjecture~\ref{conjecture:BAB} holds
in dimension $n$, then it also holds in all
dimensions~\mbox{$k\le n$}.

\smallskip
The main purpose of this paper is to
show that modulo Conjecture~\ref{conjecture:BAB}
the answer to Question~\ref{question:Jordan}
is positive even in the more general
setting of rationally connected
varieties (see Definition~\ref{definition:RC}),
and moreover the corresponding constant may be chosen
in some uniform way.
Namely, we prove the following.

\begin{theorem}
\label{theorem:RC-Jordan}
Assume that Conjecture~\ref{conjecture:BAB} holds
in dimension $n$.
Then there is a constant~\mbox{$J=J(n)$} such that
for any rationally connected
variety $X$ of dimension $n$ defined over
an arbitrary (not necessarily algebraically closed)
field $\Bbbk$ of characteristic~$0$
and for any finite subgroup $G\subset\Bir(X)$ there exists
a normal abelian subgroup $A\subset G$ of index at most $J$.
\end{theorem}

Note that Conjecture~\ref{conjecture:BAB}
is settled in dimension~$3$ (see~\cite{KMMT-2000}), so
we have the following
\begin{corollary}
The group $\Cr_3(\Bbbk)$ is Jordan.
\end{corollary}

As an application of the method we use to prove Theorem~\ref{theorem:RC-Jordan},
we can also derive some information about $p$-subgroups of Cremona groups.

\begin{theorem}\label{theorem:p-groups}
Assume that Conjecture~\ref{conjecture:BAB} holds
in dimension $n$.
Then there is a constant~\mbox{$L=L(n)$} such that 
for any rationally connected
variety $X$ of dimension $n$ 
defined over
an arbitrary (not necessarily algebraically closed)
field $\Bbbk$ of characteristic~$0$
and for any prime $p>L$, every finite
$p$-subgroup of $\Bir(X)$ is an abelian 
group generated by at most $n$ elements.
\end{theorem}
 
\begin{remark}
An easy consequence of Theorem~\ref{theorem:p-groups}
is that if $\Bbbk$ is an algebraically closed fields
of characteristic $0$, and $m>n$ are positive integers,
then there does not exist embedding of groups 
$\Cr_m(\Bbbk)\subset\Cr_n(\Bbbk)$. Indeed, for any $p$ it is easy 
to construct an abelian $p$-group $A\subset\GL_m(\Bbbk)\subset\Cr_m(\Bbbk)$
that is not generated by less than $m$ elements.
Note that the same result is already known by~\cite[\S1.6]{Demazure70}
or~\cite[Theorem~B]{Cantat-PGL}.
\end{remark}

\smallskip
The plan of the proof of Theorem~\ref{theorem:RC-Jordan}
(that is carried out in Section~\ref{section:Jordan})
is as follows.
Given a rationally connected variety $X$ and a finite group $G\subset\Bir(X)$,
take a smooth regularization $\tilde X$
of $G$ (see~\mbox{\cite[Theorem~3]{Sumihiro-1974}}).
We are going to show that $\tilde{X}$ has a point $P$ fixed
by a subgroup $H\subset G$ of bounded index
and then apply Theorem~\ref{theorem:Jordan}
to $H$ acting in the tangent space $T_P(\tilde X)$.
If $\tilde{X}$ is a $G$-Mori fiber space
(see Section~\ref{section:preliminaries} for a definition), then,
modulo Conjecture~\ref{conjecture:BAB},
we may assume that there is a \emph{non-trivial} $G$-Mori
fiber space structure $\tilde{X}\to S$, i.\,e. $S$ is not a point.
By induction one may suppose that there is a subgroup
$H$ of bounded index that fixes a point in $S$.
Using the results of Section~\ref{section:RC}
(that are based on the auxiliary results of
Section~\ref{section:preliminaries}), we show that $\tilde{X}$
contains a $G$-invariant rationally connected subvariety.
Furthermore, the same assertion holds for an arbitrary smooth $\tilde{X}$;
this follows from the corresponding assertion for
a $G$-Mori fiber space obtained by running a $G$-Minimal Model
Program on $\tilde{X}$ by the results of Section~\ref{section:RC}.
Using induction in dimension once again we conclude that
there is actually a point in $\tilde{X}$ fixed by $H$.

The main technical result that allows us to prove
Theorem~\ref{theorem:RC-Jordan} is Corollary~\ref{corollary:G-contraction}
that lets us lift $G$-invariant rationally connected
subvarieties along $G$-contractions.
Actually, it has been essentially proved
in~\cite[Corollary~1.7(1)]{HaconMcKernan07}.
The only new feature that we really need is the action of a finite
group. Since this forces us to rewrite the statements and the proofs
in any case, we use the chance to write down the details
of the proof that were only sketched by the authors
of~\cite{HaconMcKernan07}. We also refer a reader to~\cite{Kollar-Ax-2007}
and~\cite{Hogadi-Xu-2009} for ideas of similar flavour.

\smallskip
\emph{Acknowledgements.} 
We would like to thank J.-P.\,Serre who attracted our attention
to the questions discussed in this paper.
We are also grateful to I.\,Cheltsov, O.\,Fujino, S.\,Gorchinskiy and A.\,Kuznetsov
for useful discussions, and to a referee for interesting comments.
A part of this work was written during the first author's stay at the
\mbox{Max-Planck-Institut} f\"ur Mathematik in Bonn.
He would like to thank MPIM for hospitality and support.
Both authors were partially supported by
RFBR grants~\mbox{No.~11-01-00336-a}, \mbox{11-01-92613-KO\_a},
the grant of Leading Scientific Schools No.~5139.2012.1,
and AG Laboratory SU-HSE, RF~government
grant ag.~11.G34.31.0023.
The first author was partially supported by
Simons-IUM fellowship.
The second author was partially supported by the grants
\mbox{MK-6612.2012.1} and~\mbox{RFBR-12-01-33024}.

\section{Preliminaries}
\label{section:preliminaries}

The purpose of this section is to establish several
auxiliary results that will be used in Section~\ref{section:RC}.
It seems that most of them are well known to experts,
but we decided to include them for completeness since
we did not manage to find proper references.

Throughout the rest of the paper we use the standard language of the
singularities of pairs (see \cite{Kollar-1995-pairs}).
By \emph{strictly log canonical} singularities
we mean log canonical singularities that are not Kawamata log
terminal.
By a general point of a (possibly reducible) variety~$Z$
we will always mean a point in a Zariski open dense subset of $Z$.
Whenever we speak about the canonical class, or the singularities
of pairs related to a (normal) reducible variety,
we define everything componentwise
(note that connected components of a normal variety are
irreducible).

Let $X$ be a~normal variety, let $B$ be an effective $\Q$-divisor on $X$
such that the \mbox{$\Q$-divisor} $K_X+B$ is $\Q$-Cartier.
A subvariety $Z\subset X$
is called \emph{a center of non Kawamata log terminal singularities}
(or \emph{a center of non-klt singularities}) of the log pair~\mbox{$(X, B)$} 
if~\mbox{$Z=\pi(E)$} for some
divisor $E$ on some log resolution $\pi:\hat{X}\to X$
with discrepancy $a(X, B; E)\leqslant -1$.
A subvariety $Z\subset X$
is called \emph{a center of non log canonical singularities}
of the log pair~\mbox{$(X, B)$} if $Z$ is an image of some
divisor with
discrepancy strictly less than~$-1$ on some log resolution.
A center of non-klt singularities $Z$ of the log pair
$(X, B)$ is called \emph{minimal} if no other center
of non-klt singularities of
$(X, B)$ is contained in $Z$.

\begin{remark}\label{remark:finite-number}
In general it is not enough to consider one log resolution
to detect all centers of non-klt singularities of a log pair, but
the \emph{union} of these centers can be figured out using one
log resolution. Note that this
does not mean that there is only a finite number of centers
of non-klt singularities of a given log pair! Actually, the
latter happens if and only if the log pair is log canonical.
\end{remark}

Suppose that there is an action
of some finite group $G$ on
$X$ such that $B$ is $G$-invariant.
Let $Z_1$ be a center of non-klt singularities of the pair
$(X, B)$, let $Z_1, \ldots, Z_r$
be the \mbox{$G$-orbit} of the subvariety $Z_1$, and put
$Z=\bigcup Z_i$. We say that $Z$ is \emph{a $G$-center
of non-klt singularities of the pair $(X, B)$},
and call $Z$ a \emph{minimal $G$-center of non-klt singularities}
if no other $G$-center
of non-klt singularities of the pair
$(X, B)$ is contained in $Z$.
Note that one has~\mbox{$Z_i\cap Z_j=\varnothing$} for $i\neq j$ and
each $Z_i$ is normal (see \cite[1.5--1.6]{Kawamata1997}).

Suppose that $X$ is a variety with only Kawamata log terminal singularities
(in particular, this includes the assumptions that $X$ is normal and
the Weil divisor $K_X$ is
$\Q$-Cartier). A \emph{$G$-contraction}
is a $G$-equivariant proper morphism
$f:X\to Y$ onto a normal variety~$Y$ such that $f$
has connected fibers and $-K_X$ is $f$-ample
(thus $f$ is not only proper but projective).
The variety $X$ is called \emph{a $G$-Mori fiber space}
if $X$ is projective and there exists a \mbox{$G$-contraction}
$f:X\to Y$ with
$\dim(Y)<\dim(X)$
and the relative $G$-equivariant Picard number $\rho^G(X/Y)=1$.
Furthermore, if $Y$ is a point, then $X$ is called
a \emph{$G$-Fano variety}.

Suppose that $X$ is projective and $G\Q$-factorial,
i.\,e. any $G$-invariant $\Q$-divisor on $X$ is \mbox{$\Q$-Cartier}.
If $X$ is rationally connected (see Definition \ref{definition:RC}),
then one can run a $G$-Minimal Model Program on $X$,
as well as its relative versions,
and end up with a $G$-Mori fibre space.
This is possible due to~\cite[Corollary 1.3.3]{BCHM} and
\cite[Theorem~1]{MiyaokaMori}, since rational connectedness
implies uniruledness.
Actually, \cite{BCHM} treats the case when $G$ is trivial,
but adding a finite group action does not make a big difference.

\medskip

We start with proving some auxiliary statements
that will be used in course of the proof of
Theorem~\ref{theorem:RC-Jordan}.

Suppose that $V$ is a normal (irreducible) variety,
and $f:V\to W$ is a proper morphism. Then for any curve $C\subset V$ contracted
by $f$ and any Cartier divisor $D$ on $V$ one has a well-defined 
intersection index $D\cdot C$, and one can consider a (finite dimensional)
$\R$-vector space $\N_1(V/W)$ generated by the classes of curves 
in the fibers of $f$ modulo numerical equivalence 
(see e.\,g.~\cite[\S 0-1]{KMM}).

The following observation
(see e.\,g. the proofs of \cite[Theorem~1.10]{Kawamata1997} and
\cite[Theorem~1]{Kawamata-1998})
is sometimes called \emph{the perturbation
trick}.

\begin{lemma}
\label{lemma:perturbation-trick}
Let $V$ be an irreducible normal quasi-projective
variety, 
and $f:V\to W$ be a proper morphism
to a variety $W$. Let $D$ be an effective $\Q$-Cartier $\Q$-divisor
on $V$ such that
the log pair $(V, D)$ is strictly log canonical.
Suppose that a finite group $G$ acts on $V$ so that $D$ is
$G$-invariant.
Let $Z\subset X$ be a minimal
\mbox{$G$-center} of non-klt singularities of the log pair $(V, D)$.

Choose $\varepsilon>0$ and
a compact subset \mbox{$\K\subset\N_1(V/W)$}.
Then there exists
a $G$-invariant $\Q$-Cartier $\Q$-divisor
$D'$ such that
\begin{itemize}
\item the only centers of non-klt singularities of the log pair
$(V, D')$ are the irreducible components of $Z$;
\item for any $\kappa\in\K$ one has
$|(D-D')\cdot \kappa|<\varepsilon$.
\end{itemize}
\end{lemma}
\begin{proof}
Let $Z_1$,\ldots, $Z_r$ be irreducible components of $Z$.
Note that $Z_i$'s are disjoint by~\mbox{\cite[Proposition~1.5]{Kawamata1997}}.
Let $\mathscr M$ be a linear system of very ample
divisors such that $\operatorname{Bs}\mathscr M=Z_1$
and let $M_1\in \mathscr M$ be a general element.
Let $M_1,\ldots, M_l$ be the $G$-orbit of~$M_1$, and let~\mbox{$M=\sum M_i$}.
For $0<\theta \ll 1$ the subvariety $Z_1$ is the only center of
non log canonical singularities for the log pair $(V, D+\theta M_1)$.
Hence the only centers of non log canonical
singularities for $(V, D+\theta M)$
are the subvarieties $Z_i$. Now take $\delta\in\Q_{{}>0}$ so that
the log pair $(V, D')$ is strictly log canonical,
where $D'=(1-\delta)D+\theta M$.
By the above the only centers of non-klt
singularities of $(V, D')$ are $Z_i$'s.
Since $\theta\ll 1$, one has~\mbox{$\delta\ll 1$}, 
which guarantees the existence 
of an appropriate~$\varepsilon$.
\end{proof}

\begin{remark}\label{remark:perturbation-trick-generalizations}
One can generalize Lemma~\ref{lemma:perturbation-trick}
assuming that we start from a log pair that includes any formal linear
combination of linear systems on the variety $V$ with rational coefficients
instead of a divisor $D$, and produce an effective $\Q$-divisor $D'$.
Another version of the same assertion produces a movable linear system
$\mathcal{D}'$ instead of a divisor $D'$. Note that neither Lemma~2.2
nor these generalizations require the morphism~$f$ to be equivariant with
respect to the group~$G$. 
\end{remark}

We will need the following Bertini-type statement.

\begin{lemma}[{cf.~\cite[Theorem~1.13]{Reid-1980can}}]
\label{lemma:Bertini}
Let $Z$ be a normal variety and $D$ be an effective $\Q$-divisor
on $Z$ such that the log pair $(Z, D)$ is Kawamata
log terminal. Let $\mathscr M$ be a base point free linear system and
let $M\in \mathscr M$ be a general member. Then
\begin{itemize}
\item the variety $M$ is normal;
\item the log pair $(M, D|_M)$ is Kawamata log terminal.
\end{itemize}
\end{lemma}
\begin{proof}
Doing everything componentwise, we may assume that $Z$
is connected. Since $Z$ is normal, it is irreducible.
The pair $(Z,D+\mathscr M)$ is purely log terminal
(see Definition~4.6 and Lemma~4.7.1 in~\cite{Kollar-1995-pairs}).
Hence $(Z,D+M)$ is also purely log terminal
(see \cite[Theorem~4.8]{Kollar-1995-pairs}).
Thus by the inversion of adjunction
(see e.g. \cite[5.50--5.51]{Kollar-Mori-1988}) the variety~$M$
is normal and the pair
$(M,D|_{M})$ is Kawamata log terminal.
\end{proof}

The following is a relative version of the usual Kawamata subadjunction
theorem.

\begin{lemma}\label{lemma:relative-subadjunction}
Let $V$ be an irreducible normal quasi-projective variety,
and $D$ be an effective $\Q$-divisor
on $V$ such that
the log pair $(V, D)$ is strictly log canonical.
Let $W$ be a normal quasi-projective
variety, and $f:V\to W$ be a proper morphism
with connected fibers such that
$-(K_V+D)$ is $f$-ample.

Suppose that $G$ is a finite group acting on $V$.
Let $Z\subset V$ be a minimal
$G$-center of non-klt singularities of the log pair $(V, D)$,
and~\mbox{$T=f(Z)\subset W$}.
Let $Z_t=Z\cap f^{-1}(t)$ be a fiber of $f\vert_Z$ over a general point
$t\in T$.
Then
\begin{itemize}
\item $Z_t$ is normal;
\item $Z_t$ is irreducible;
\item there exists an effective
$\Q$-divisor $D_{Z}$ on $Z$ such that $K_{Z}+D_{Z}$
is $\Q$-Cartier, the log pair $(Z_t, D_Z\vert_{Z_t})$
is Kawamata log terminal and
$$
K_{{Z}_t}+D_{{Z}\vert_{Z_t}}\qlin (K_{{V}}+{D})\vert_{{Z}_t}.
$$
\end{itemize}
\end{lemma}
\begin{proof}
By Lemma~\ref{lemma:perturbation-trick}
we may assume that $Z$ is the only $G$-center
of non-klt singularities of the log pair $(V, D)$.
Furthermore, since an intersection of centers of non-klt
singularities is again a center of non-klt singularities
(see~\cite[Proposition~1.5]{Kawamata1997}),
we conclude that each of the connected components
of $Z$ is irreducible, because otherwise
the pairwise intersections of irreducible components of $Z$
would be a (non-empty) union of $G$-centers of non-klt
singularities of the pair $(V,D)$.
Applying~\cite[Theorem~1.6]{Kawamata1997} to connected components
of $Z$, one obtains that $Z$ is normal
(note that connected components
of $Z$ are minimal centers of non-klt singularities of $(V, D)$).
Moreover, a general fiber~$Z_t$
is connected by the Nadel--Shokurov connectedness theorem
(see e.\,g.~\cite[Theorem~3.2]{Corti2000}).
Hence~$Z_t$ is irreducible.
 
To proceed we may drop the action of the group $G$ and
assume that $T$ is a point.
Indeed, let $W'\subset W$ be a general hyperplane section,
and $t\in W'$ be a general point
(which is the same as to choose $t$ to be a general point of $W$,
and then to choose a general hyperplane section $W'\ni t$).
Put $V'=f^{-1}(W')$.
By Lemma~\ref{lemma:Bertini} the variety~$V'$ is normal.
Let $\varphi: \tilde V\to V$
be a log resolution of $(V,D)$, and let $\tilde V'$ be the
proper transform of $V'$. Since~$V'$ is a general member
of a base point free
linear system, $\varphi$ is also a log resolution of the log pair~\mbox{$(V,D+V')$}.
Therefore, $\varphi$ induces a log resolution of $(V', D|_{V'})$.
This implies that the log pair $(V', D|_{V'})$
is log canonical and the irreducible components of
$Z'=Z|_{V'}$ are its minimal
centers of non-klt singularities.
Replacing $f:(V,D)\to W$ by 
$$f|_{V'}: (V', D|_{V'})\to W'$$
and repeating this process $\mathrm{codim}_{W}(T)$ times,
we get the situation where $T$ is a point, and~\mbox{$Z=Z_t$}
(in particular, $Z$ is projective, normal, and irreducible).

With these reductions done, we apply Kawamata's subadjunction theorem
(see e.\,g.~\cite[Theorem~1]{Kawamata-1998}
or~\cite[Theorem~1.2]{FujinoGongyo12})
to conclude that there exists an effective
$\Q$-divisor $D_{{Z}}$ on ${Z}$ such that $K_{{Z}}+D_{{Z}}$
is $\Q$-Cartier, the log pair $(Z, D_Z)$ is Kawamata log terminal and
$$K_{{Z}}+D_{{Z}}\qlin (K_{{V}}+{D})\vert_{{Z}}.$$
\end{proof}

\begin{remark}
A usual form of the Kawamata's subadjunction theorem
(as in~\cite{Kawamata-1998} and~\cite{FujinoGongyo12})
requires the ambient variety to be projective.
Therefore, if one wants to be as accurate
as possible, the end of the proof of Lemma~\ref{lemma:relative-subadjunction}
should be read as follows. Assuming that~$T$ is a point,
we know that $Z$ is projective; as above, we can also suppose that
$Z$ is the only center of non-klt singularities of $(V, D)$.
Taking a log canonical closure
$(\bar{V}, \bar{D})$ of the log pair $(V, D)$ as in~\cite[Corollary~1.2]{HaconXu},
we see that $Z$ is still a minimal center of
non-klt singularities of the new pair $(\bar{V}, \bar{D})$,
and all other centers of
non-klt singularities of $(\bar{V}, \bar{D})$
are disjoint from $Z$. Now~\cite[Theorem~1.2]{FujinoGongyo12}
implies the assertion of Lemma~\ref{lemma:relative-subadjunction}.
Since this step is more or less obvious, we decided not to
include it in the proof to save space (and readers attention) for more essential
points.

Another interesting moment in the proof of
Lemma~\ref{lemma:relative-subadjunction} that we want to emphasize
is that we do not care about the action of the group $G$ anywhere apart from
the equivariant perturbation trick at the very beginning
(in particular, the morphism~$f$ is not required to be $G$-equivariant,
cf.~Remark~\ref{remark:perturbation-trick-generalizations}).
On the other hand, it seems that one cannot replace this
$G$-perturbation by a non-equivariant perturbation
performed at some later step, since otherwise we would not know
that the fiber $Z_t$ is connected, and thus it would
remain undecided if we have occasionally got
rid of some components of $Z$ or not.
This is crucial for us, since we are going to obtain
a $G$-invariant subvariety $Z$ with controllable fibers.
\end{remark}

\section{Rationally connected subvarieties}
\label{section:RC}
In this section we develop techniques to ``pull-back''
invariant rationally connected subvarieties under
contractions appearing in the Minimal Model Program.
Basically we follow the ideas of~\cite{HaconMcKernan07}.

Recall the following standard definitions.

\begin{definition}[{see e.\,g.~\cite[Lemma-Definition 2.6]{Prokhorov-Shokurov-2009}}]
A (normal irreducible) variety $X$ is called a \emph{variety of Fano type}
if there exists an effective $\Q$-divisor $\Delta$ on $X$ such that
the pair $(X,\Delta)$ is Kawamata log terminal and
$-(K_X+\Delta)$ is nef and big.
\end{definition}

\begin{definition}[{see e.\,g.~\cite[\S IV.3]{Kollar-1996-RC}}]
\label{definition:RC}
An irreducible variety $X$ is called \emph{rationally connected}
if for two general
points $x_1, x_2\in X$ there exists a rational map
\mbox{$t:C\dasharrow X$}, where~$C$ is a rational curve, such that the image
$t(C)$ contains $x_1$ and $x_2$.
\end{definition}

In particular, a point is a rationally connected variety.
Furthermore, rational connectedness is birationally
invariant, and an image of a rationally connected
variety under any rational map is again rationally connected.

The following is an easy consequence
of Lemma~\ref{lemma:relative-subadjunction}.

\begin{lemma}\label{lemma:RC-fibers}
Let $f:V\to W$ be a $G$-contraction
from a quasi-projective variety $V$ with Kawamata log terminal singularities.
Choose an effective $G$-invariant $\Q$-Cartier $\Q$-divisor~$D_W$ on $W$,
and put $D=f^*D_W$.

Let $Z\subset V$ be a minimal
$G$-center of non-klt singularities of the log pair $(V, D)$,
and~\mbox{$T=f(Z)\subset W$}.
Let $Z_t=Z\cap f^{-1}(t)$ be a fiber of $f\vert_Z$ over a general point
$t\in T$. Then~$Z_t$ is a variety of Fano type.
In particular, $Z_t$ is rationally connected.
\end{lemma}
\begin{proof}
By Lemma~\ref{lemma:relative-subadjunction}
a general fiber $Z_t$ is a normal irreducible variety
(so that we may assume $\dim(Z_t)>0$), and there exists an effective
$\Q$-divisor $D_{Z}$ on $Z$ such that $K_{Z}+D_{Z}$
is $\Q$-Cartier, the log pair $(Z_t, D_Z\vert_{Z_t})$ is Kawamata log terminal
and
$$K_{Z_t}+D_{Z}\vert_{Z_t}\qlin
(K_{V}+D)\vert_{Z_t}.$$

Since $Z_t$ is an irreducible variety such that 
the restriction of $D$ to $Z_t$ is trivial,
and the restriction of $-K_V$ to $Z_t$ is ample, we see that 
$Z_t$ is a variety of Fano type.
The last assertion of the lemma follows
from~\cite[Theorem~1]{Zhang-Qi-2006} 
or~\cite[Corollary~1.13]{HaconMcKernan07}.
\end{proof}

Now we are ready to prove the main technical result of this section.

\begin{lemma}[{cf.~\cite[Corollary~1.7(1)]{HaconMcKernan07}}]
\label{lemma:G-contraction}
Let $f:V\to W$ be a $G$-contraction
from a quasi-projective variety
with Kawamata log terminal singularities onto a quasi-projective variety~$W$.
Let \mbox{$T\subsetneq W$} be a $G$-invariant irreducible
subvariety. Then there exists a $G$-invariant (irreducible)
subvariety $Z\subsetneq V$ such that $f\vert_Z:Z\to T$ is
dominant and a general fiber of $f\vert_Z$ is rationally connected.
\end{lemma}
\begin{proof}
Take $k\gg 0$, and choose $H_1,\ldots,H_k$ to be general divisors
from some (very ample) linear system $\mathcal{H}$ with $\Bs\mathcal{H}=T$.
Adding the images of the divisors $H_i$ under the action
of~$G$ to the set $\{H_1,\ldots,H_k\}$ we may assume that this set
is $G$-invariant. Put \mbox{$D_W=\sum H_i$}, and $D=f^*D_W$. Let $c$ be
the log canonical threshold of the log pair $(V, D)$ over a general point
of $T$ (this makes sense
since the log canonical threshold in a neighborhood of a point $P\in V$
is an upper semi-continuous
function of $P$, and $T$ is irreducible).
Note that we can assume
that for any center $L$ of non-klt singularities of $(V, cD)$ one has
$f(L)\subset T$ by the construction of $D$.

Let $S$ be a union of all centers of non-klt singularities
of the log pair $(V, cD)$ that
\emph{do not} dominate $T$.
Then $T$ is not contained in the set $f(S)$.
Indeed, the union $\mathcal{Z}$ of centers of non-klt singularities
of $(V, cD)$ is a union of a finite number
of centers of \mbox{non-klt} singularities of $(V, cD)$ by
Remark~\ref{remark:finite-number}.
By definition of $c$ we conclude that
the log pair~\mbox{$(V, cD)$} has a center of non-klt singularities
$Z_1$ that dominates~$T$.

Let $Z_1, \ldots, Z_r$ be the $G$-orbit of the subvariety $Z_1$, and put
$Z=\bigcup Z_i$.
Put \mbox{$W^o=W\setminus f(S)$} and $V^o=f^{-1}(W^o)$,
and note that $Z\cap V^o$ is a minimal $G$-center of non-klt singularities of
the log pair $(V^o, cD\vert_{V^o})$.
Lemma~\ref{lemma:RC-fibers} implies that the fiber $Z_t$ of
the morphism
$f\vert_Z$ over a general point $t\in T\cap W^o$ is rationally connected.
\end{proof}

\begin{remark}
In the case when $f$ is an isomorphism over a general point
of $T$ the proof of Lemma~\ref{lemma:G-contraction}
produces the strict transform of $T$ on $V$ as a
resulting subvariety $Z$.
\end{remark}

Rationally connected varieties enjoy the following important property
(see~\cite[Corollary~1.3]{Graber-Harris-Starr-2003} for the proof over~$\mathbb{C}$; the case
of an arbitrary field of characteristic~$0$ follows by the usual Lefschetz
principle).

\begin{theorem}
\label{theorem:GHS}
Let $f:X\to Y$ be a dominant morphism of proper
varieties over $\Bbbk$.
Assume that both $Y$ and a general fiber of $f$ are rationally connected.
Then $X$ is also rationally connected.
\end{theorem}

Together with the previous considerations this
enables us to lift $G$-invariant rationally connected varieties
via $G$-contractions.
Namely, the following immediate
consequences of Lemma~\ref{lemma:G-contraction}
and Theorem~\ref{theorem:GHS}
will be used in the proof of Theorem~\ref{theorem:RC-Jordan}.

\begin{corollary}\label{corollary:G-contraction}
Let $f:V\to W$ be a $G$-contraction
from a quasi-projective variety
with Kawamata log terminal singularities onto a quasi-projective variety~$W$.
Let $T\subsetneq W$ be a $G$-invariant rationally connected
subvariety. Then there exists a $G$-invariant rationally connected
subvariety $Z\subsetneq V$ that dominates $T$.
\end{corollary}
\begin{proof}
Apply Lemma~\ref{lemma:G-contraction}
to obtain a subvariety $Z\subsetneq V$ that maps to a rationally connected
variety $T$ with a rationally connected general fiber.
Theorem~\ref{theorem:GHS} applied to (a~desingularization of a
compactification of) $Z$ completes the proof.
\end{proof}

The following is just a small modification of
Corollary~\ref{corollary:G-contraction}, but we find it useful to state it
to have a result allowing us to lift rationally connected
subvarieties via (equivariant) flips.

\begin{corollary}\label{corollary:G-flip}
Let $f:V\to W$ be a $G$-contraction
from a quasi-projective variety
with Kawamata log terminal singularities onto a quasi-projective variety~$W$.
Consider a diagram of $G$-equivariant morphisms
$$
\xymatrix{
V\ar@{->}[rd]_{f}&&V'\ar@{->}[ld]^{f'}\\
&W&
}
$$
Suppose that there exists a $G$-invariant rationally connected
subvariety \mbox{$Z'\subset V'$}
such that~\mbox{$f'(Z')\neq W$}.
Then there exists a $G$-invariant rationally connected
subvariety $Z\subsetneq V$.
\end{corollary}
\begin{proof}
Apply Corollary~\ref{corollary:G-contraction}
to the rationally connected variety $T=f'(Z')\subsetneq W$.
\end{proof}

Corollaries~\ref{corollary:G-contraction} and~\ref{corollary:G-flip}
imply the following assertion.

\begin{lemma}\label{lemma:G-MMP}
Let $V$ be a projective
variety with an action of a finite group $G$.
Suppose that~$V$ has Kawamata log terminal $G\Q$-factorial singularities.
Let $f:V\dasharrow W$ be a birational map that is a result of a
$G$-Minimal Model Program ran on $V$.
Let $F\subset G$ be a subgroup.
Suppose that there exists an $F$-invariant rationally connected
subvariety~\mbox{$T\subsetneq W$}.
Then there exists an $F$-invariant rationally connected
subvariety $Z\subsetneq V$.
\end{lemma}
\begin{proof}
Induction in the number of steps of the $G$-Minimal Model Program
using Corollaries~\ref{corollary:G-contraction} and~\ref{corollary:G-flip}
(note that any $G$-contraction is also an $F$-contraction).
\end{proof}

In particular, Lemma~\ref{lemma:G-MMP} implies the following assertion
(it will not be used directly
in the proof of our main theorems, but still we suggest that
it deserves being mentioned).

\begin{proposition}\label{lemma:G-resolution}
Let $W$ be a quasi-projective variety
with terminal singularities acted on by a finite group $G$
so that $W$ is $G\Q$-factorial.
Let $f:V\to W$ be a $G$-equivariant resolution of singularities
of $W$.
Suppose that there exists a $G$-invariant rationally connected
subvariety $T\subsetneq W$.
Then there exists a $G$-invariant rationally connected
subvariety~\mbox{$Z\subsetneq V$}.
\end{proposition}
\begin{proof}
Run a relative $G$-Minimal Model Program
on $V$ over $W$
(this is possible due to an equivariant version
of~\cite[Corollary~1.4.2]{BCHM})
to obtain a variety $V_n$ that is a relatively minimal model over $W$ together
with a series of birational modifications
$$V=V_0\stackrel{f_1}\dashrightarrow
\ldots\stackrel{f_n}\dashrightarrow
V_n\stackrel{f_{n+1}}\longrightarrow W.$$
Then $f_{n+1}$ is
small by the Negativity Lemma (see e.\,g.~\cite[2.19]{Utah}).
Thus $G\Q$-factoriality of $W$ implies that $f_{n+1}$
is actually an isomorphism.
Now the assertion follows from Lemma~\ref{lemma:G-MMP}.
\end{proof}

\section{Jordan property}
\label{section:Jordan}

In this section we will prove Theorem~\ref{theorem:RC-Jordan}.
Before we proceed let us introduce the following notion.

\begin{definition}\label{definition:almost-fixed}
Let $\CC$ be some set of varieties.
We say that \emph{$\CC$ has almost fixed points}
if there is a constant $J=J(\CC)$ such that
for any variety $X\in\CC$ and for any finite
subgroup~\mbox{$G\subset\Aut(X)$} there exists
a subgroup $F\subset G$ of index at most $J$
acting on $X$ with a fixed point.
\end{definition}

Theorem~\ref{theorem:RC-Jordan}
will be implied by the following auxiliary result.

\begin{theorem}\label{theorem:RC-fixed-point}
Let $\RR(n)$ be the set of all rationally connected
varieties of dimension $n$. Assume that Conjecture~\ref{conjecture:BAB}
holds. Then $\RR(n)$ has almost fixed points.
\end{theorem}

\begin{remark}
In the proof of Theorem~\ref{theorem:RC-Jordan}
we will only use the particular case of Theorem~\ref{theorem:RC-fixed-point}
for smooth rationally connected varieties.
However, it is more convenient to prove
it without any assumptions on singularities.
In any case, it does not make a big difference
(see Corollary~\ref{corollary:non-singular} below).
\end{remark}

Sometimes it would be convenient to restrict
ourselves to non-singular varieties when proving
assertions like Theorem~\ref{theorem:RC-fixed-point}.
It is possible by the following (nearly trivial)
observation.

\begin{lemma}\label{lemma:non-singular}
Let $\CC$ be some set of varieties, and let $\CC'\subset\CC$.
Suppose that for any
$X\in\CC$ and for any finite group $G\subset\Aut(X)$ there is a
variety $X'\in\CC'$ with $G\subset\Aut(X')$ and a
$G$-equivariant surjective morphism
$X'\to X$. Then
$\CC$ has almost fixed points if and only if $\CC'$ does.
\end{lemma}
\begin{proof}
An image of a fixed point under an equivariant morphism is again
a fixed point.
\end{proof}

\begin{corollary}\label{corollary:non-singular}
The set $\RR(n)$ of rationally connected varieties
of dimension $n$ has almost fixed points
if and only if the set $\RR'(n)$ of non-singular
rationally connected varieties does.
\end{corollary}

To prove Theorem~\ref{theorem:RC-fixed-point}
we will need its particular case concerning Fano varieties.

\begin{lemma}\label{lemma:Fano-fixed-point}
Let $\FF(n)$ be the set of all Fano varieties
of dimension $n$ with terminal singularities,
and assume that Conjecture~\ref{conjecture:BAB}
holds in dimension $n$. Then~$\FF(n)$ has almost fixed points.
\end{lemma}
\begin{proof}
Using Noetherian induction,
one can show that there exists a positive integer $m$ such that
for any $X\in \FF(n)$ the divisor $-mK_X$
is very ample and gives
an embedding 
$$X\hookrightarrow\mathbb{P}^{\dim |-mK_X|}.$$
So we may assume that any $X\in \FF(n)$ admits
an embedding $X \hookrightarrow \mathbb P^N$ for some $N=N(n)$
(that does not depend on $X$)
as a subvariety of degree at most $d=d(n)$.
Moreover, the action of $G\subset \Aut(X)$ is induced by an action
of some linear group $\Gamma\subset \operatorname{GL}_{N+1}(\mathbb C)$.
By Theorem~\ref{theorem:Jordan} there exists an abelian subgroup
$\Gamma_0\subset \Gamma$ of index at most $I=I(N+1)$.
Let $G_0\subset G$
be the image of $\Gamma_0$ under the natural projection from $\Gamma$ to
$G$. Take linear independent
$\Gamma_0$-semi-invariant sections
$$s_1,\ldots,s_{N+1}\in H^0(X,-mK_X).$$
They define $G_0$-invariant
hyperplanes $H_1,\ldots, H_{N+1}\subset \mathbb P^N$.
Let $k$ be the minimal positive integer such that
$$X\cap H_1\cap \ldots \cap H_k=\{P_1,\ldots,P_r\}$$
is a finite ($G_0$-invariant) set. Then $r\leqslant d$.
Since the stabilizer $G_1\subset G_0$ of $P_1$
is a subgroup of index at most $r\le d$,
the assertion of the lemma follows.
\end{proof}

Lemma~\ref{lemma:Fano-fixed-point}
allows us to derive a slightly wider
particular case of Theorem~\ref{theorem:RC-fixed-point}
involving $G$-Mori fiber spaces
from the assertion of Theorem~\ref{theorem:RC-fixed-point}
for lower dimensions.

\begin{lemma}\label{lemma:MF-fixed-point}
Suppose that the sets $\mathcal{R}(k)$ of rationally connected
varieties of dimension $k$ have almost fixed points for
$k\le n-1$, and assume that Conjecture~\ref{conjecture:BAB}
holds in dimension~$n$. Then
there is a constant $J=J(n)$ such that for any finite group
$G$ and for any rationally connected
$G$-Mori fiber space $\phi:M\to S$ with $\dim(M)=n$
there is a finite subgroup of index at most $J$ in $G$ acting on
$M$ with a fixed point.
\end{lemma}
\begin{proof}
Let $\phi:M\to S$ be a rationally connected
$G$-Mori fiber space of dimension $n$.
We are going to show that there is a constant $J$
that does not depend on $M$ and $G$ such that
there exists
a subgroup $H\subset G$ of index at most $J$
acting on $M$ with a fixed point.
By Lemma~\ref{lemma:Fano-fixed-point}
we may suppose that $1\le\dim(S)\le n-1$.

Consider an exact sequence of groups
$$
1\to G_{\phi}\longrightarrow G\stackrel{\theta}
\longrightarrow G_{S}\to 1,
$$
where the action of $G_{\phi}$ is fiberwise with respect to $\phi$ and
$G_S$ is the image of $G$ in $\Aut(S)$.
Note that $S$ is rationally connected since so is $M$.
By assumption there is a constant~$J_1$
that does not depend on $S$ and $G$ such that
there exists
a subgroup $F_S\subset G_S$ of index at most~$J_1$
acting on $S$ with a fixed point. Let $P\in S$ be one of the points fixed
by $F_S$.

Define a subgroup $F\subset G$ to be the preimage of the subgroup
$F_S\subset G_S$ under the homomorphism $\theta$.
Then $\phi:M\to S$ is an $F$-contraction.
By Corollary~\ref{corollary:G-contraction} applied to the group $F$
and the contraction $\phi$ there exists
an $F$-invariant rationally connected subvariety
$Z\subset M$ such that $\phi(Z)=P$. In particular, $\dim(Z)<n$.
By assumption there is a constant $J_2$ that
does not depend on $Z$ and $F$ such that
there is a subgroup $H\subset F$ of index at most $J_2$
acting on $Z$ (and thus on $X$) with a fixed point. The assertion follows
since $[G:F]=[G_S:F_S]\le J_1$.
\end{proof}

Now we are ready to prove Theorem~\ref{theorem:RC-fixed-point}.

\begin{proof}[{Proof of Theorem~\textup{\ref{theorem:RC-fixed-point}}}]
Let $X$ be a non-singular
(or terminal) rationally connected variety of dimension
$n$, and $G\subset\Aut(X)$ be a finite subgroup.
By Corollary~\ref{corollary:non-singular}
it is enough to show that there is a constant $J$
that does not depend on $X$ and $G$ such that
there exists
a subgroup $H\subset G$ of index at most $J$
acting on $X$ with a fixed point.

Run a $G$-Minimal Model Program on $X$,
resulting in a $G$-Mori fiber space $X'$ and a rational map
$f:X\dasharrow X'$ that factors into a sequence of
$G$-contractions and $G$-flips.
By Lemma~\ref{lemma:MF-fixed-point} there is a constant $J_1$
that does not depend on $X'$ (and thus on $X$) and $G$
such that there exists
a subgroup $F\subset G$ of index at most $J_1$
acting on $X'$ with a fixed point.
Using Lemma~\ref{lemma:G-MMP}
applied to the group $F$, we obtain an $F$-invariant rationally connected
subvariety $Z\subsetneq X$.

The rest of the argument is similar to that in the proof of
Lemma~\ref{lemma:MF-fixed-point}.
Using induction in $n$, we see that there is a constant $J_2$ that
does not depend on $Z$ and $F$ such that
there is a subgroup $H\subset F$ of index at most $J_2$
having a fixed point on $Z$ (and thus on $X$),
and the assertion of the theorem follows.
\end{proof}

\begin{remark}
To prove Theorem~\ref{theorem:RC-fixed-point}
one could actually use a weaker version of Lemma~\ref{lemma:MF-fixed-point}.
For this purpose it is sufficient to know that the $G$-Mori fiber space
contains an $F$-invariant rationally connected subvariety $Z'\subsetneq X'$
for some subgroup $F\subset G$ of bounded index, without assuming that
$Z'$ is a point.
\end{remark}

Now we are going to derive Theorem~\ref{theorem:RC-Jordan}
from Theorem~\ref{theorem:RC-fixed-point}.
We will need the following easy observation.

\begin{lemma}
\label{lemma:normal}
Let $G$ be a group and $H\subset G$ be a subgroup of finite index~\mbox{$[G:H]=j$}.
Suppose
that $H$ has some property $\mathcal{P}$ that is preserved under
intersections and under conjugation in $G$. Then there exists
a normal subgroup $H'\subset G$ of finite index~\mbox{$[G:H']\le j^j$}
such that $H'$ also enjoys the property $\mathcal{P}$.
\end{lemma}
\begin{proof}
Let $H_1=H, \ldots, H_r\subset G$ be the subgroups that are conjugate to $H$.
Then $r\le j$, and $H'=\bigcap H_i$ is normal and has the property
$\mathcal{P}$. It remains to notice that \mbox{$[G:H']\le j^r$}.
\end{proof}

\begin{proof}[{Proof of Theorem~\ref{theorem:RC-Jordan}}]
We may assume that
the field $\Bbbk$ is algebraically closed.
Let $X$ be a rationally connected variety of dimension $n$,
and~\mbox{$G\subset\Bir(X)$} be a finite group.
Let~$\tilde{X}$ be a regularization
of $G$, i.\,e. $\tilde{X}$ is a projective variety with
an action of $G$ and
a~\mbox{$G$-equivariant} birational map $\xi:\tilde{X}\dashrightarrow X$
(see~\mbox{\cite[Theorem~3]{Sumihiro-1974}}).
Taking a $G$-equivariant resolution
of singularities (see~\cite{Abramovich-Wang}), one
can assume that $\tilde{X}$ is smooth.
Note that $\tilde{X}$ is rationally connected since so is $X$.
By Theorem~\ref{theorem:RC-fixed-point}
there is a constant~$J_1$ that does not depend on~$\tilde{X}$ and $G$
(and thus on $X$) such that there exists a
subgroup $F\subset G$ of index at most $J_1$ and
a point $P\in X$ fixed by $A$.
The action of~$F$ on the
Zariski tangent space $T_P(\tilde{X})\cong\Bbbk^{n}$ is faithful
(see e.\,g.~\cite[Lemma~2.7(b)]{FlennerZaidenberg}).
By Theorem~\ref{theorem:Jordan} applied to~\mbox{$\GL_n(\Bbbk)$}
there is a constant $J_2$ (again independent of anything
except for $n$) such that~$F$ has an abelian subgroup $A$ of index
at most $J_2$. The assertion follows by Lemma~\ref{lemma:normal}.
\end{proof}

Finally, we prove Theorem~\ref{theorem:p-groups}.

\begin{proof}[{Proof of Theorem~\ref{theorem:p-groups}}]
We may assume that
the field $\Bbbk$ is algebraically closed.
Let $X$ be a rationally connected variety of dimension $n$, and 
let~\mbox{$G\subset\Bir(X)$} be a finite $p$-group. Arguing as in
the proof of Theorem~\ref{theorem:RC-Jordan}, we obtain
an abelian subgroup $F\subset G$ of index~\mbox{$[G:F]$} 
bounded by some constant $L$
(that does not depend on $X$ and $G$) with an embedding
$F\subset\GL_n(\Bbbk)$. The latter implies that the
abelian $p$-group $F$ is
generated by at most $n$ elements. On the other hand, if~\mbox{$p>L$}, then
the index of any subgroup of $G$ is at least $p$, so that
the subgroup $F$ coincides with $G$.
\end{proof}

\end{document}